\documentclass[a4paper,12pt]{article}
\usepackage{graphicx,epsfig,cite}
\usepackage{amsfonts}
\usepackage{amsmath}
\usepackage{amssymb}
\usepackage{longtable}
\usepackage{setspace}
\usepackage[usenames]{color}
\usepackage{colortbl}
\usepackage{indentfirst}
\usepackage{amsthm}
\newtheorem{lemma}{Lemma}  
\newtheorem{theorem}{Theorem}

\newlength{\abc}
\title{\bf On Cross Hyperoperatorial Migration \\ of Properties, Related to Natural \\ Number Division Operator}
\author{V.~Sh.~Tlyusten\footnotemark[1], V.~B.~Tlyachev\footnotemark[2]\\[2mm]
\small \footnotemark[1]~Adygea State University, 385000, Maykop,
Pervomayskaya str. 208, \\ \small 385000, Russia, val\_t@mail.ru \\
\small \footnotemark[2]~Adygea State University, 385000, Maykop,
Pervomayskaya str. 208, \\ \small 385000, Russia, tlyachev@adygnet.ru}
\date{}

\begin{document}
\maketitle
\begin{center}

\begin{abstract}
In the article integer divisibility properties and related prime
factors natural number representation concepts have been defined
over the whole infinite hyperoperation hierarchy. The definitions
have been made across and above of unique arithmetic operations,
composing this hierarchy (addition, multiplication,
exponentiation, tetration and so on). It allows the habitual
concepts of ``prime factor``, "multiplier", "divider", "natural
number factors representation"  etc., to be associated mainly with
the same sense, with the each of those operations. As analogy of
multiplication-based Fundamental Theorem of Arithmetic (FTA), an
exponentiation-based theorem is formulated. The theorem states
that any natural number $M$ can be uniquely represented as a
tower-like exponentiation: $M=a_n\uparrow(a_{n-1}\uparrow (\ldots
(a_2\uparrow a_1)\ldots )),$ where $a_i\neq 1 (i=1,\ldots,n)$ are
primitive in some sense (related to the exponentiation operation),
exponentiation components, following one by one in some unique
order and named in the article as biprimes. 
\end{abstract}

\end{center}

\section{Introduction}

\normalsize
\indent  The English mathematician R.L.~Goodstein is well
known as the author of an interesting and original approach to
constructive mathematical analysis. This approach differs
significantly in the general idea and nature of the basic concepts
from the approaches used by other mathematicians. Goodstein's
approach is closely related to the calculus of equality that he
introduced, which is an axiomatic fragment of the theory of
recursive arithmetic functions, which have some important
advantages. In particular, he used only such algorithms that, by
their definition, always finish work. Functions specifically
designed for this purpose, such as the Goodstein hyperoperator
sequence, extend the basic arithmetic operations well beyond
exponentiation.Goodstein's theoretical constructions are presented
in [1-3]. These works contain a systematic and detailed research
of the Goodstein calculus of primitive-recursive equalities and
research of some modifications of this calculus. With the
Goodstein sequence of hyperoperations, the method of denoting very
large integers, introduced by Donald Knuth in 1976 [4], so called
Knuth's up-arrow notation, is related. This article is devoted to
some problems arising in the theory of Goodstein.

We will start with some auxiliary definitions and denotations.

Let $\mathbb{N}$ be the set of all natural integers,
$\mathbb{N}_0=\mathbb{N}\cup \{0\},$ and $H_r(a,x)$ be a
\textit{hyperoperator} of rank $r (a, x, r \in \mathbb{N}_0),$
defined by following recursive schema:
\begin{equation}\label{eq:num-1}
\begin{array}{ll}
H_0(a,x)=a+x,\\
H_{r+1}(a,0)=1, \, H_{r+1}(a,1)=a,\\
H_{r+1}(a,x+1)=H_r(a,H_{r+1}(a,x)), \, x\geq 1.
\end{array}
\end{equation}

From this definition follows, that for natural numbers, the
notation of \textit{hyperoperator} generalizes and extends the
notation of \textit{arithmetic addition operator}. And it allows
building, on addition operation basis, the infinite sequence of
other homogeneously defined arithmetic operations of arbitrary
higher ranks (as a members of the sequence they can be called
\textit{hyperoperations}). The structure we have chosen for
definition (1) corresponds to the following beginning of
rank-ordered hierarchy of these operations:
\begin{equation}\label{eq:num-2}
H_0(a,x)=a+x, \, H_1(a,x)=a\times x, \, H_2(a,x)=a^x, \ldots .
\end{equation}

Inverse function $H_r^{-1}(a,y)$ for any element $H_r(a,x)$ of
sequence (2) can be defined in such a way that
$H_r^{-1}(a,H_r(a,x)))=x.$ Then, for example, assuming $r=0, 1,
2,$ we get conventional arithmetic operations:
\textit{subtraction} $(y-a),$ \textit{integer division} $(y \; div \;
a)$ and \textit{computing discrete logarithm} $(\log_a y),$
correspondently.


\textbf{Definition 1.} At being given a natural number sequence
(vector) $\overline{\upsilon}=(\upsilon_1,\ldots, \upsilon_n)$ and
any $r\in \mathbb{N}_0,$ we will call $r$--\textit{multiplicator}
(or, multiplicator of rank $r$) of vector $\overline{\upsilon}$ an
operator $\prod^{/r/},$ represented by the following formula:
\begin{equation}\label{eq:num-3}
\Pi^{/r/}\overline{\upsilon}=\Pi_{i=1,n}^{/r/}{\upsilon}_i=
H_r(\upsilon_n,H_r(\upsilon_{n-1},\ldots ,
H_r(\upsilon_2,H_r(\upsilon_1, Sgn(r)))\ldots )).
\end{equation}

Talking about the right part of the formula (3) and assuming $r =
0, 1, 2, \ldots,$ we will also be calling it
$r$--\textit{multiplication} (thus, generalizing by this term
concepts of addition, multiplication, exponentiation and so on).

The significance of the $r$--\textit{multiplication} notion is
that it allows considering hyperoperatorial operations for
operand's vectors of an arbitrary none zero length. It should be
noted also, that operands binding in formula (3) is carried out
from right to left. That's essential in cases when $r>1$ and
associative property may not be true (in \textit{hyperoperator
hierarchy} (2) it's always true for addition and multiplication
only).

Examples:
\begin{equation}
\begin{array}{ll}
\Pi_{i=1,1}^{/0/}{\upsilon}_i=H_0(\upsilon_1,
Sgn(0)))=\upsilon_1+0=\upsilon_1,\\
\Pi_{i=1,3}^{/1/}{\upsilon}_i=H_1(\upsilon_3,H_1(\upsilon_2,H_1(\upsilon_1,Sgn(1))))=\upsilon_3\times(\upsilon_2\times(\upsilon_1\times 1))=\upsilon_3\times(\upsilon_2\times\upsilon_1),\\
\Pi_{i=1,3}^{/2/}{\upsilon}_i=H_2(\upsilon_3,H_2(\upsilon_2,H_2(\upsilon_1,Sgn(2))))=\upsilon_3\uparrow(\upsilon_2\uparrow(\upsilon_1\uparrow
1))=\upsilon_3^{(\upsilon_2^{\upsilon_1})}.
\end{array}\nonumber
\end{equation}

\section{The rank's cross over divisibility}


\textbf{Definition 2.} Let $r\in\mathbb{N}_0$  be the rank of some
hyperoperator $H_r(a,x).$ Any natural number $M > Sgn(r)$ is
called $r$--\textit{decomposable\textit} ($r$--\textit{compound})
if the equation
\begin{equation}\label{eq:num-4}
H_r(a,x)=M
\end{equation}
has a solution $(a_0,x_0)$  over $\mathbb{N},$ where $a_0, x_0
>Sgn(r).$

Otherwise (that is, if there are no such solutions of that
equation), we will call given $M$ $r$--\textit{prime}.

Particularly, in cases when $r = 0, 1, 2,$ we can talk about
additive, multiplicative and exponential \textit{decomposability}
(or, on the contrary, \textit{quantity} to be prime) of number
$M,$ correspondently. In addition, for convenience, instead of
term "2-prime" we will often use in this article a new special
term --- "biprime".

When $r=1$ equation (4) is specified as $a\times x=M,$
$(M,a,x>1),$ and the notion of multiplicative decomposability (or,
on the contrary, of quantity to be $1$-prime) coincides with usual
notion of decomposability (or, on the contrary, of quantity to be
prime) of given number $M.$

For $r = 0$ we will get from (4) the equation $a+x=M,$ $(M,a,x>0),$
which corresponds to trivial statement about additive
decomposability of any natural number $M$ (the only 0-prime in
this case is 1).

Notion of $r$--decomposability of given natural number $M >
Sgn(r),$ particularly, in cases $r = 0, 1, 2,$ is equal by sense
to possibility of representation of this number, correspondently,
as sum, product, power of two non trivial (for production and
exponentiation -- not equal to 1) natural numbers.

The only $0$--prime number, as already noted, is 1;  number 6, for
example, is a biprime (equation $x^y = 6$ has the trivial solution
(6,1) only), while the number 9 can serve as an example of a
bi-decomposable number (since $3^2 = 9$).

Incidentally, we also note that any \textit{prime number} is
\textit{biprime} (the converse statement is obviously not true).


\textbf{Definition 3.} For a given hiperoperator's rank $r,$ two
different natural numbers $a, b$ we will call
$r$--\textit{coprimes} if there isn't exist a natural $d >Sgn(r)$
such that system of equations:
\begin{equation}\label{eq:num-5}
\begin{cases}
H_r(d,x)=a,\\
H_r(d,y)=b,\\
x,y\geq Sgn(r)
\end{cases}
\end{equation}
has at least one solution $(x_0,y_0)$ over $\mathbb{N}$ with
respect to the variables $x, y.$

If $a,b$ are not $r$--\textit{coprimes,} that is, if there exists
a natural number $d >Sgn(r)$ such that system (5) is resolvable
with respect to the variables $x,y,$ then the number $d$ will be
called \textit{common} $r$--\textit{divisor} of $a$ and $b.$

At last, if $d$ is \textit{common} $r$--\textit{divisor} of $a$
and $b,$ wherein $(x_0,Sgn(r))$ is a solution of (5) with respect
of $x, y$ (in that case $d = b$), then $d$ is called
$r$--\textit{divisor} of $a;$ $x_0$ is called quotient of
$r$--\textit{divisor} $a$ by $d$ ($r$--\textit{quotient}) and so
on.

For example, number $3$ is $0-,$ $1-,$ $2-$ \textit{divisor} of
numbers $5,$ $6,$ $9,$ correspondently (since $\textbf{3}+2=5,$
$\textbf{3}\times 2=6,$ $\textbf{3}^2=9$).

The concepts of $r$--\textit{remainder, greatest common}
$r$--\textit{divisor, common} $r$--\textit{multiple} and some
others can be defined in a similarly manner, by extending  of well
known notions, terms and, possibly, some statements, associated
with multiplying and divisibility (that the same, 1-multiplying
and 1-divisibility) of natural numbers to the case of their
r-multiplying  and r-divisibility, where $r\neq 1.$

No doubt that among all categories listed above, the most
interesting one is the category of statements. In fact, it would
be interesting to know what well known theorems on the
divisibility of natural numbers remain valid (or, at least, can be
transformed into similar in meaning their analogues) when going
over to hyperoperator ranks different from the multiplicative one.

The rest of this paper is devoted to the first step in the
direction of possible considerations of this issue. There we "will
try" the Fundamental Theorem of Arithmetic (FTA) to variant of an
exponential decomposition of natural numbers (that is, we will
translate the statement of FTA from the hyperoperator's rank $r=1$
to the rank  $r=2$).

We will conclude the section with the following general remark,
concerning the system of notation we have adopted.

Talking about the properties of $r$--operators, everywhere, if
possible, we will be using the standard mathematical notation,
adding upper $/\textbf{r}/$ marks, where they are needed, to
conventional mathematical designations.  Thus, for instance, on
the base of using of symbol "$|$" as a designator of predicate
"divides", the previous examples of $r$--divisibility can be
written, correspondently, in the form of statements:
$\textbf{3}|^{/0/}5,$ $\textbf{3}|^{/1/}6,$ $\textbf{3}|^{/2/}9.$

\section{The exponential decomposability}

Following the generalized hyperoperator paradigm of the
consideration of basic arithmetic operations, adopted in this
article, as an example of possible cross rank boundary migration
of not only concepts, but also facts and theorems, related to
natural numbers divisibility, let's ask ourselves the question:

If, in accordance with the FTA, any natural number greater than 1
either is a prime itself or can be represented as the product of
prime factors, so that such a representation is unique, up to
(except for) the order of the factors, then what can be said about
the existence of an analogous representation of any natural number
as the bi-product of biprime factors?

The answer to this question is given by the following theorem.

\begin{theorem} \textbf{Theorem (on bi-decomposability).}
Any natural number $M$ greater than 1 either is a \textit{biprime}
itself or can be uniquely represented as a tower-like
exponentiation: $M=a_n\uparrow(a_{n-1}\uparrow(\ldots (a_2\uparrow
a_1)\ldots)),$ where $a_i\neq 1 (i=1,\ldots , n)$ -- are following
one by one in some unique order \textit{biprime} exponentiation
components.
\end{theorem}

Postponing for a while the proof of this theorem, we will first
formulate and prove the two lemmas.

\begin{lemma}
Let $D$ is a bi-compound natural integer. Then there are exist
$d_0$--biprime and $D_0\in \mathbb{N} (D_0>1)$ such, that:
\begin{equation}\label{eq:num-6}
D=d_0^{D_0}.
\end{equation}
\end{lemma}

\begin{proof}[Proof of Lemma 1.]
So far as $D$ is a bi-compound natural integer, it can be
represented as:
\begin{equation}\label{eq:num-7}
D=d_1^{D_1}.
\end{equation}
where $d_1,$ $D_1\in \mathbb{N},$ $d_1>1,$ $D_1>1.$ Here two cases
are possible:

1) Number $d_1$ is biprime. Then in the formula (6) we set
$d_0=d_1,$ $D_0=D_1$ and stop.  For this case the lemma is proved.

2) Number $d_1$ is bi-compound. Then for some $c_2\in\mathbb{N}$
$(c_2>1)$ we have $d_1=d_2^{c_2},$ where $1<d_2<d_1.$

Substituting this expression for $d_1$ in (7) we obtain:
\begin{equation}\label{eq:num-8}
D=d_2^{c_2^{D_1}}=d_2^{D_2}.
\end{equation}

If $d_2$ is biprime then in (6) we assume $d_0=d_2,$ $D_0=D_2,$
and lemma is proved.  Otherwise there again can be found $c_3 \in
\mathbb{N}$ $(c_3 > 1)$ and $1<d_3<d_2,$ such that
$d_2=d_3^{c_3}.$ Having substituted $d_2$ in (8), we will obtain
$D=d_3^{c_3^{D_2}}=d_3^{D_3}.$

Continuing this process further and constructing a sequence of
decreasing but staying greater than $1$ elements $d_1,$
$d_2,\ldots ,$ we will necessarily stop and come to some element
$d_s,$ which turns out to be biprime.

The final result of this process is the exponent
$D=d_s^{c_sD_{s-1}}=d_s^{c_2c_3...c_sD1}=d_s^{D_s},$ where is
$d_s$ biprime, $D_s>1.$ Supposing at the exponent $d_0=d_s,$
$D_0=D_s$ we will obtain from it the expression (6).

The lemma 1 is proved for this case too.
\end{proof}

\begin{lemma} For arbitrarily chosen
different biprimes $a,$ $b,$ the exponential Diophantine equation
$a^x=b^y$ is not solvable with respect to  $x, y.$
\end{lemma}

\begin{proof}[Proof of Lemma 2.]
In the conditions of the above restrictions on $a$ and $b,$ let us
admit the opposite, that is, assume that the equation $a^x=b^y$
has at least one solution $(x_0,y_0).$ We have:
\begin{equation}\label{eq:num-9}
a^{x_0}=b^{y_0}=C.
\end{equation}

Let $C = \Pi_{i=1}^n p_i^{c_i},$ where $p_i$ --- different prime
numbers, and $c_i (c_i\in \mathbb{N})$ are their powers. Then,
taking in consideration 9, we can write down:
\begin{equation}\label{eq:num-10}
a = \Pi_{i=1}^n p_i^{\alpha_i}, b = \Pi_{i=1}^n p_i^{\beta_i},
\end{equation}
\begin{equation}\label{eq:num-11}
\Rightarrow a^{x_0}= \Pi_{i=1}^n p_i^{\alpha_i x_0}, b^{y_0} =
\Pi_{i=1}^n p_i^{\beta_i y_0},
\end{equation}
\begin{equation}\label{eq:num-12}
\Rightarrow  \alpha_i x_0=\beta_i y_0, \, (i=1,2,...,n).
\end{equation}

After denoting the least common multiple of numbers $x_0, y_0$  as
$I=LCM(x_0,y_0),$ from (11, 12) can be deduced: $\alpha_i
x_0=Ik_i$, $\beta_i y_0=Ik_i,$ where $I = qx_0,$ $I = ry_0$ and
$q,$ $r,$ $k_i\in \mathbb{N}.$ That is, we have: $\alpha_i x_0=q
x_0k_i$ $\Rightarrow \alpha_i=qk_i;$ $\beta_i y_0=ry_0k_i$
$\Rightarrow \beta_i=rk_i.$

Substituting the last expressions for $\alpha_i$ and $\beta_i$
into (10) gives:
\begin{equation}\label{eq:num-13}
a= \Pi_{i=1}^n p_i^{q k_i}, \, b=\Pi_{i=1}^n p_i^{r k_i}.
\end{equation}

At last, denoting $d=\Pi_{i=1}^n p_i^{k_i},$ from the formula
(13), we deduce $a=d^{\, q},$ $b=d^{\,r},$ which contradicts the
initial assumption that $a,$ $b$ are biprimes. This contradiction
proves the lemma 2.
\end{proof}

\begin{proof}[Proof of the Theorem 1.]
Let $M>1$ be a natural number mentioned in the hypothesis of the
Theorem 1. Here we are going to show that there exists a
bi-factorization of $M$ into a sequence of biprime components.

If $M$ is biprime, then our goal is trivially reached. Otherwise,
by lemma 1, we can write: $M=M_0=a_1^{M_1},$ where $a_1$ is
biprime, $M_1\in \mathbb{N},$ $1<M_1<M_0.$

If $M_1$ is biprime, then the required sequence of biprime
components is found. Otherwise, we will write down again:
$M_1=a_2^{M_2},$ where $a_2$ is biprime, $M_2\in \mathbb{N},$
$1<M_2<M_1.$

Continuing this process further, we obtain a decreasing sequence
of natural numbers $M_0>M_1>M_2>... ,$ which, because of its
boundedness bellow, at some point breaks off on the element $M_s$
such that $M_s>1$ and $M_s$ is non bi-compound. Thus $M_s$ turns
out to be biprime. Renaming $M_s=a_s,$ we get the desired
bi-factorization:
\begin{equation}\label{eq:num-14}
M=a_1\uparrow (a_2\uparrow(\ldots (a_{s-1}\uparrow a_s)\ldots )).
\end{equation}
or, the same in common notation (3):
\begin{equation}\label{eq:num-15}
M= \Pi_{i=s,1}^{/2/} a_i=H_2(a_1, H_2(a_2,\ldots ,H_2(a_{s-1},
H_2(a_s, Sgn(2)))\ldots )).
\end{equation}

2. Now we will show the uniqueness of bi-factorization (14). Suppose
it's not unique, that is, can be found at least two different
sequences of primes $(a_1,a_2,\ldots , a_{s_1})$ and
$(b_1,b_2,\ldots ,b_{s_2})$ such, that:
\begin{equation}\label{eq:num-16}
M=a_1\uparrow (a_2\uparrow(\ldots (a_{s_1-1}\uparrow
a_{s_1})\ldots ))=b_1\uparrow (b_2\uparrow(\ldots
(b_{s_2-1}\uparrow b_{s_2})\ldots )).
\end{equation}

Let $t = min (s_1,s_2).$ The following cases may be occurred:

A) $\exists e (1\leq e\leq t, a_e\neq b_e)$ and $\forall i
(i=\overline{1,e-1}, a_i=b_i).$ In this case we obtain from (16)
either the equality $a_e^{M_1}=b_e^{M_2}$ ($a_e,$ $b_e$ are
biprimes, $M_1>1,$ $M_2>1$), which, by the hypothesis of lemma
2, is impossible, or one of the two equalities: $a_e^M=b_e,$ or
$b_e^M=a_e$ ($a_e,$ $b_e$ are biprimes, $M>1$). The last two
equalities, due to the fact that a biprime cannot coincide with a
bi-compound number, are obviously also not possible.

B) $a_i=b_i,$ $(i=\overline{1,t})$. In this case, from (16) we
obtain a numerical equality of the form:
\begin{equation}\label{eq:num-17}
a_1\uparrow (a_2\uparrow(\ldots (a_t\uparrow M)\ldots
))=a_1\uparrow (a_2\uparrow(\ldots (a_{t-1}\uparrow a_t)\ldots )),
\,  where \, M>1.
\end{equation}
But the last equality is, obviously, also not true (value of
expression in the left part is greater than value of expression in
the right one).

Thus, the equality (16) is possible only if the expression of its
left part exactly coincide with the expression of its right part.
The proof of the uniqueness of bi-factorization and the theorem as
a whole are completed.
\end{proof}

\section{Conclusion, Open Problems}

In this article we have considered the possibility of expanding
the system of concepts based on the multiplication and division of
natural numbers, in such a way as it would cover the entire
infinite hierarchy of hyperoperatorial arithmetic operations. As
an example and object for the first step, illustrating the
author's approach to such an extension, a hyperoperator of
exponentiation was chosen.

Basic notions for this case are introduced: a \textit{bi-product}
and \textit{biprime} ones. In particular, the latter is understood
to mean a natural number that cannot be represented as the degree
of two nontrivial (not equal to 1) natural numbers.

For an arithmetical hyperoperator of exponentiation, as an analog
of the Fundamental Theorem of Arithmetic, a theorem on exponential
(tower-like) decomposability of natural numbers is formulated and
proved. The theorem states that any natural number (not equal to
1) is either \textit{biprime} or can be uniquely represented as a
\textit{bi-product} of following in some order \textit{biprimes.}

From the theorem, in particular, it follows (in the article this
fact is stated and proved in the form of a lemma), that for
arbitrarily chosen, different biprimes $a,$ $b,$ the exponential
Diophantine equation $a\uparrow x = b\uparrow y$ is not solvable
with respect to the $x,$ $y.$

There are two main directions on which this study could be
continued.

1. Further development of the theory of bi-decomposability and
biprime numbers; search for possible applications of this theory.

2. Extending of system of those the classical concepts, facts and
statements related to the multiplication and divisibility of
natural numbers that admit their natural interpretation throughout
the hierarchy of hyperoperatorial arithmetic operations.

In connection with the second of the two directions of research
mentioned above, it would be interesting, in particular, the issue
of the validity of the following generalized version of the proved
in this article theorem (here we will express this version in the
form of a hypothesis which it seems to be correct).

\textit{Hypothesis (on $r$--decomposability)}. For a given natural
$r\neq 1,$ any natural number $M>1$ is either $r$--simple in
itself or it can be uniquely represented as an $r$--product of
$r$--prime components that follow in a certain order.

\end{document}